\documentclass{amsart}
\usepackage{amsthm,amsmath,amsfonts,amssymb,enumerate,array,multirow,booktabs,stmaryrd,esint}
\usepackage{tikz}
\usepackage{hyperref}
\usetikzlibrary{matrix,arrows,shapes}

\newtheorem{thmintro}{Theorem}
\newtheorem{thm}{Theorem}[section]
\newtheorem{lemma}[thm]{Lemma}

\newtheorem{prop}[thm]{Proposition}

\newtheorem{problem}{Problem}

\theoremstyle{definition}

\theoremstyle{remark}

\newtheorem{oss}[thm]{Remark}

\newcommand{\bnabla}{\bar \nabla}
\newcommand{\C}{\mathbb C}
\newcommand{\R}{\mathbb R}
\newcommand{\Z}{\mathbb Z}
\newcommand{\CP}{\mathbb{CP}}
\newcommand{\RP}{\mathbb{RP}}
\newcommand{\lie}{\mathfrak}
\renewcommand{\O}{\Ort}

\DeclareMathOperator{\tr}{tr}
\DeclareMathOperator{\Aut}{Aut}
\DeclareMathOperator{\U}{U}
\DeclareMathOperator{\SO}{SO}
\DeclareMathOperator{\G}{G}
\DeclareMathOperator{\Ort}{O}
\DeclareMathOperator{\id}{id}
\DeclareMathOperator{\Rm}{Rm}
\DeclareMathOperator{\Ric}{Ric}
\DeclareMathOperator{\const}{const}
\DeclareMathOperator{\vol}{vol}
\DeclareMathOperator{\dive}{div}
\renewcommand{\div}{\dive}
\renewcommand{\phi}{\varphi}
\renewcommand{\epsilon}{\varepsilon}

\renewcommand{\bar}{\overline}
\renewcommand{\tilde}{\widetilde}

\newcommand{\de}{\partial}

\newcommand{\call}{\mathcal}
\DeclareMathOperator{\Spin}{Spin}
\DeclareMathOperator{\Defo}{Def}

\title{On minimal Legendrian submanifolds of Sasaki-Einstein manifolds} 
\author{Simone Calamai and David Petrecca} \thanks{Both authors are members of ``Gruppo Nazionale per le Strutture Algebriche, Geometriche e le loro Applicazioni''.}

\address{(S.~Calamai) Dip. di Matematica e Informatica ``U. Dini'' - Universit\`a di Firenze \endgraf Viale Morgagni 67A -  Firenze - Italy}
\email{simocala at gmail.com}
\address{(D.~Petrecca) Dipartimento di Matematica - Universit\`a di Pisa \endgraf Largo Pontecorvo 5 -  Pisa - Italy}
\email{petrecca at mail.dm.unipi.it}

\begin{document}
\maketitle
\begin{abstract}
For a given minimal Legendrian submanifold  $L$ of a Sasaki-Einstein manifold we construct two families of eigenfunctions of the Laplacian of $L$ and we give a lower bound for the dimension of the corresponding eigenspace. Moreover, in the case the lower bound is attained, we prove that $L$ is totally geodesic and a rigidity result about the ambient manifold. This is a generalization of a result for the standard Sasakian sphere done by L\^e and Wang.
\end{abstract}

\section*{Introduction}
Let $(M, \eta, g)$ be a Sasakian manifold of dimension $2n+1$. A minimal Legendrian submanifold is an $n$-dimensional submanifold $i: L \rightarrow M$ on which the contact form vanishes, $i^*\eta = 0$ and is minimal in the sense of Riemannian geometry with respect to the metric induced from $g$.

In the case where the minimal Legendrian $L$ is embedded in the standard Sasakian round $(2n+1)$-sphere, L\^e and Wang \cite{lewang} constructed a family of functions on $L$ which are eigenfunctions of the Laplacian on $L$ of the induced metric. They give also a lower bound of the dimension of the relative eigenspace and if it is attained then the submanifold is totally geodesic. Conversely they prove that a minimal submanifold of the standard sphere admitting that certain family of functions as Laplacian eigenfunctions is necessarily Legendrian.

Although their techniques make a heavy use of the particular situation, namely the theory of minimal immersion in spheres and the presence of an ambient Euclidean space,  we prove that some of their ideas can be generalized to any Sasaki-Einstein manifold.

Let $L$ be a minimal Legendrian submanifold of a Sasaki-Einstein $M$. The aim of this paper is to prove that two certain families of functions on $L$, one of which constructed in terms of the contact moment map of the action of the Sasaki automorphism group, are eigenfunctions of the Laplacian of $L$ and we give a lower bound for the dimension of the eigenspace.
\begin{thmintro}
Let $L^n$ be a minimal Legendrian submanifold of an $\eta$-Sasaki-Einstein manifold $(M^{2n+1}, \eta, \xi, g, \Phi)$ with algebra of infinitesimal Sasaki automorphisms $\lie g$. Then, for each $X \in \lie g$, the function
\[
\eta(X) - \frac 1 {\vol(L)} \int_L \eta(X)dv,
\]
where $dv$ is the volume form of $L$ of the induced metric, is an eigenfunction of the Laplacian $\Delta_L$ with eigenvalue $2n+2$. Moreover the dimension of the $(2n+2)$-eigenspace is at least $\dim \lie g - \frac 1 2 n( n+1) -1$.
\end{thmintro}

Moreover we prove, like in the sphere case although with totally different techniques, that if the lower bound is attained then the submanifold is totally geodesic together with a  rigidity result about the ambient $M$, in the case of a \emph{regular} Sasaki-Einstein manifold over a base K\"ahler manifold.

\begin{thmintro} If $M$ is a regular Sasaki-Einstein manifold and the multiplicity  of the eigenvalue $2n+2$ of $\Delta_L$ is exactly $\dim \lie g - \frac 12 n(n+1)-1$ then $M$ is a Sasaki-Einstein circle bundle over the complex projective space endowed with the Fubini-Study metric. In particular if $M$ is simply connected then $M = S^{2n+1}$.
\end{thmintro}

Among the techniques we use we mention the theory of deformations of minimal Legendrian submanifolds, for which we refer to \cite{ono, ohnita} and, in the case of regular manifolds, the correspondence between Legendrian submanifolds of Sasakian manifolds and Lagrangian submanifolds of K\"ahler manifolds, see \cite{reckziegel}.

This result makes use of the geometry of Legendrian submanifolds of the K\"ahler-Einstein base, which exists by the regularity assumption. It would be interesting to drop this assumption and prove the result for quasi-regular or irregular Sasaki-Einstein manifolds. 

Then in Theorem \ref{thmeigenfunction} we provide a generalization of the family of eigenfunctions by making use of the immersion of the Sasaki-Einstein manifold $M$  into its Ricci-flat K\"ahler cone $C(M)$. This family is parameterized by the Lie algebra of the infinitesimal K\"ahler automorphisms of $C(M)$, which is in general bigger than the Sasaki automorphism group of $M$.
The family is defined by means of the Nomizu operator on $C(M)$. This time our arguments are similar to the ones of L\^e and Wang for the sphere and they rely on the Ricci-flatness of $C(M)$ and  properties of the Nomizu operator. 

It would be interesting to provide sufficient conditions for the Legendrianity of a minimal submanifold by means of any of these families of functions.

\begin{problem}
Let $M^{2n+1}$ be a Sasaki-Einstein manifold with big enough automorphism group $G$, let $L^n$ be a minimal submanifold such that for each $X \in \lie g$, the family of functions \eqref{eqetaX} or \eqref{equa:defnefK} are eigenfunctions of $\Delta_L$ with eigenvalue $2n+2$. Can we conclude that $L$ is Legendrian?
\end{problem}

Also, it would be interesting to relate the second family with the moment map of the symplectic action on $C(M)$ of its K\"ahler automorphism group.

The paper is organized as follows. In section \ref{sec:prelim} we recall some notions from Sasakian geometry, minimal Legendrian deformations and the contact moment map. In section \ref{sec:etaX} we introduce the first family of eigenfunctions and prove the main theorems. Finally in section \ref{sec:nomizu} we construct the functions via the Nomizu operator.

\subsection*{Acknowledgements}
The authors would like to thank Xiuxiong Chen for constant support and Fabio Podest\`a 
for suggesting the problem and his help and advice. We finally thank Anna Gori for useful discussions.

\section{Preliminaries} \label{sec:prelim}
We recall some notions from Sasakian geometry, minimal Legendrian submanifolds and their deformations.
\subsection{Sasakian manifolds}
In this paper we focus on the case where the contact manifold is a Sasakian manifold, i.e. there is a contact form $\eta$, its Reeb field $\xi$, a Riemannian metric $g$ and a $(1,1)$-tensor field $\Phi$ such that
\begin{equation*}
\begin{split}
\eta(\xi) = 1, \iota_\xi d\eta = 0 \\
\Phi^2 = - \id + \xi \otimes \eta \\
g(\Phi \cdot, \Phi \cdot) = g + \eta \otimes \eta \\
d\eta = g( \Phi \cdot, \cdot) \\
N_\Phi + \xi \otimes d\eta = 0
\end{split}
\end{equation*}
where $N_\Phi$ is the torsion of $\Phi$.

An equivalent formulation is to say that a Riemannian manifold $(M, g)$ is Sasakian if and only if its symplectization $C(M)=M \times \R^+$ with metric $\bar g = r^2 g + dr^2$ is a K\"ahler manifold, where $r$ is the coordinate on $\R^+ = (0, +\infty)$.

A Sasakian structure defines a \emph{transverse K\"ahler structure} on $M$, that it is defines a K\"ahler structure $(d\eta, \Phi|_D)$ on the contact subbundle $D = \ker \eta$. This K\"ahler metric is known as the \emph{transverse metric}.

The Reeb field $\xi$ is unitary and Killing and defines a foliation called \emph{characteristic foliation}. We call $M$ \emph{regular} Sasakian if the circle action defined by the characteristic foliation is free. It is known, see e.g. \cite{monoBG, blair} that every compact regular Sasakian manifold is a Riemannian submersion $\pi:M \rightarrow B$ over a compact K\"ahler manifold.
In the \emph{quasi-regular} case, i.e. when the circle action is locally free, we have an orbifold Riemannian submersion.

A Sasaki-Einstein metric is a Sasakian metric which is Einstein, i.e. its Ricci tensor is a multiple of the metric. By curvature properties of Sasakian metrics, see i.e. \cite{monoBG}, it follows that this constant is $2n$, where $2n+1$ is the dimension of $M$.
An $\eta$-Sasaki-Einstein metric is a Sasakian metric $g$ such that its Ricci tensor satisfies $\Ric = A g + (2n-A) \eta \otimes \eta$ for some constant $A$. The following proposition is well known.

\begin{prop}[\cite{sparks_survey}]
Let $(M, g)$ be a Sasakian manifold. Then $g$ is Sasaki-Einstein if and only if the transverse metric is K\"ahler-Einstein with constant $2n+2$, if and only if the K\"ahler cone $C(M)$ is Ricci-flat.
\end{prop}

\subsection{Legendrian immersions and their deformations}

We will consider some special submanifolds of Sasakian manifolds, known as \emph{Legendrian} (or \emph{horizontal}), see \cite{reckziegel}.

A Legendrian submanifold of a $(2n+1)$-dimensional contact manifold $(M, \eta)$ is an $n$-dimensional submanifold $i: L \rightarrow M $ such that for all $p \in L$ we have $i_*(T_p L) \subseteq \ker \eta_{i(p)}$.

We will consider Legendrian submanifolds which are also \emph{minimal} in the sense of Riemannian geometry, i.e. their mean curvature field vanishes. 

If we have a Legendrian submanifold $L$ in a Sasakian manifold we can identify the space of sections of the normal bundle $NL$ with $C^\infty(L) \oplus \Omega^1(L)$ via the isomorphism
\begin{align*}
\chi	: \Gamma(NL)  &\longrightarrow 	 C^\infty(L) \oplus \Omega^1(L)\\
	 V			 &\longmapsto 		\biggl ( \eta(V), -\frac 1 2 i^* (\iota_V d\eta) \biggr )
\end{align*}
see \cite{ono}.

In the case of a compact regular Sasakian manifold $M$ with contact structure $\eta$ that fibers over a compact K\"ahler manifold $(B, \omega)$ we can take the projection $\pi(L) \subseteq B$ of a Legendrian $L$.
Following Reckziegel \cite{reckziegel} we have that $\pi(L)$ is a Lagrangian submanifold of $B$, i.e. $(\pi \circ i)^* \omega = 0$ and is finitely covered by $L$. 

Conversely, given a Lagrangian submanifold $j:N \rightarrow B$, a point $q \in N$, for any choice of $p$ in the fiber of $q$ there exists a neighborhood $U$ of $q$ and a Legendrian immersion $i: U \rightarrow M$ such that $\pi \circ i = j|_U$.

Moreover, Riemannian properties of $L$ hold as well for $\pi(L)$ and conversely. Namely we have the following.
\begin{prop} [\cite{reckziegel}] \label{prop:reckziegel}
The Legendrian $L$ is minimal, or totally geodesic, if and only if the Lagrangian $\pi(L)$ is.
\end{prop}

A smooth family of minimal Legendrian immersions $i_t: L \rightarrow M$ is a family of maps $F: [0,1] \times L \rightarrow M$ such that for each $t$ the map $i_t= F(t, \cdot) : L \rightarrow M$ is a minimal Legendrian immersion. 
Every smooth family points out a vector field $W_t$  on $L$ given at $p$ by
\[
W_t |_p =  F_* \biggl( \frac \de {\de t} \biggr |_{(t, p)} \biggr ).
\]
It is known, e.g. \cite{ohnita,ono}, that a family of immersions is Legendrian if and only if the normal component $V_t$ of $W_t$ satisfies
\begin{equation} \label{infleg}
V_t = \chi^{-1} \biggl (\eta(V_t), \frac 1 2 d \eta(V_t) \biggr ),
 \end{equation}
 i.e. $d\eta(V_t) = - i^* (\iota_{V_t} d\eta)$. Normal fields satisfying \eqref{infleg} are called \emph{infinitesimal Legendrian deformations}.

We are interested in minimal Legendrian deformations of a Legendrian $i: L \rightarrow M$, that are smooth families $i_t: L \rightarrow M$ of minimal Legendrian immersions  such that $i_0=i$.

A trivial family of deformations of a minimal Legendrian submanifolds is given by one-parameter families of ambient transformations. We will denote by $\Aut(M)$ the group of such transformations, i.e. diffeomorphisms $M \rightarrow M$ which are isometric contactomorphisms.

If we let $\phi_t \in \Aut(M)$ be one of such families. Then $i_t = \phi_t |_{i(L)}  : i(L) \rightarrow M$ is a minimal Legendrian deformation, see \cite{ohnita}.

In particular, the normal component of every field in the Lie algebra $\lie{aut}(M)$ of $\Aut(M)$ defines an infinitesimal Legendrian deformation. This is also minimal as we are taking the normal component of a Killing vector field, see \cite[Sec. 3]{simons}.

When we restrict ourselves to $\eta$-Sasaki-Einstein manifolds with constant $A$, we have a characterization of the space of infinitesimal minimal Legendrian deformations.
\begin{prop}[\cite{ohnita}]
Let $i:L \rightarrow M$ be a minimal Legendrian submanifolds in an $\eta$-Sasaki-Einstein manifold with constant $A$. Then the vector space of infinitesimal minimal Legendrian deformations is identified with
\[
\Defo(L) = \R \oplus \{ f \in C^\infty(L): \Delta_L f = (A+2)f \}
\]
where $\Delta_L$ denotes the Laplacian of $L$ with the induced metric.
\end{prop}

This result is obtained by combining the copy of the space of infinitesimal Legendrian deformations of $L$ given by
\[
\biggl \{ \biggl (f, \frac 1 2 df \biggr ): f \in C^\infty(L) \biggr \}
\]
and the space of minimal deformation given by the kernel of the Jacobi operator $\call J$, for which we refer to \cite{simons}.

\subsection{Contact moment maps} \label{subsec:moment}
We finally recall the notion of contact moment map, we follow \cite[Sec.~8.4.2]{monoBG}.
In our setting the group $G = \Aut(M)$ is a compact group acting on $M$. We can extend this action to the symplectic cone $(C(M), d(r^2 \eta))$ by requiring that it leaves the $\{r = \const \}$ levels unchanged, i.e. the action is given by $g (p, r) = (gp, r)$.
Being $G$ a contactomorphism group it is easy to see that the action on $C(M)$ is by symplectomorphisms and, being the symplectic form on the cone exact, this action is Hamiltonian.
So there exists a map $\phi: C(M) \rightarrow \lie g^*$, such that
\[
d (\phi(X)) = - \iota_X d (r^2 \eta) = d (r^2 \eta(X)).
\]
Hence, up to a constant, one can take the map $\phi(p, r) (X) = r^2 \eta_p(X)$.
Seeing $M$ as the $\{r=1 \}$ level set, we consider the restriction $\mu: M \rightarrow \lie g^*$ of $\phi$ which we call the \emph{contact moment map} for the $G$-action on $M$.

\section{Eigenfunctions using the contact moment map} \label{sec:etaX}
In this section we construct one possible generalization of the functions given by L\^e-Wang \cite{lewang}.
We briefly recall their setting. They consider the standard Sasakian sphere $S^{2n+1}$ immersed in its K\"ahler cone $\C^{n+1} \backslash \{0 \}$ with respectively the round metric $g$ and the Euclidean metric $\langle \cdot, \cdot \rangle$.
It is known that the both the Sasaki transformation group of the sphere and the K\"ahler automorphism group of the cone is $G=\U(n+1)$. 
Let $M \in \lie u(n+1)$. Then the moment map for the $G$-action on the cone is given, up to a constant, by
\[
\phi(p, r)(M) = r^2 \eta_p(M_p) = r^2 g(\xi_p, M_p) =  \langle \xi_p, M_p \rangle
\]
We see an infinitesimal Sasaki automorphism $M \in \lie u(n+1)$ as a linear vector field whose value at $x \in S^{2n+1}$ is $Mx$. Then, using that $\xi$ at $x$ is $Jx$, where $J$ is the standard complex structure, the contact moment map $\mu: S^{2n+1} \rightarrow \lie u(n+1)^*$ is given by
\[
\mu(x)(M) = \langle Mx, Jx \rangle
\]
which is exactly the function of L\^e-Wang.

Back to the general setting of the Sasaki group $G= \Aut(M)$ with Lie algebra $\lie g$ acting on the $\eta$-Sasaki-Einstein $M$, we have the contact moment map that is given by $\mu(p)(X) = \eta_p(X_p)$.

We then consider for each $X \in \lie g$ the map $p \mapsto \eta(X)$ restricted to a minimal Legendrian submanifold and up to a constant. 

We prove the generalization of one of the implications of \cite[Thm. ~1.1]{lewang}.

\begin{thm} \label{thm:eigenf_etaX}
Let $(M, g, \eta, \xi)$ be a $(2n+1)$-dimensional $\eta$-Sasaki-Einstein manifold with $\Ric = A g + (2n-A) \eta \otimes \eta$ and let $L^n \subset M$ be 
a minimal Legendrian submanifold. 
Then for all $X \in \lie{aut}(M)$ the function on $L$ given by 
\begin{equation} \label{eqetaX}
f_X = \eta(X) - \frac 1 {\vol(L)} \int_L \eta(X)dv,
\end{equation}
where $dv$ is the volume form on $L$ of the induced metric, is en eigenfunction of the Laplacian $\Delta_L$ on $L$ with eigenvalue $A+2$. Moreover this eigenspace has dimension $\geq \dim \lie{aut}(M) - \frac 1 2 n(n+1)-1$.
\end{thm}

\begin{proof}
We recalled above that the map $\chi: \Gamma(NL) \rightarrow C^\infty(L) \oplus \Omega^1(L)$ given by $\chi(V) = (\eta(V), -\frac 1 2 \iota_V d\eta)$ 
is an isomorphism if $L$ is Legendrian and that the space of infinitesimal deformations of a minimal Legendrian $L$ is
\[
\Defo(L) = \R \oplus \{ f \in C^\infty(L): \Delta_L f =  (A+2) f \}.
\]

Let $X \in \lie g = \lie{aut}(M)$ and let $X|_L = X_1 + X_2 \in \Gamma(TL) \oplus \Gamma(NL)$ be its decomposition.

From \cite{ohnita} it follows that $X_2$ defines a Legendrian deformation of $L$ and it is known, e.g. \cite{simons}, 
that the normal part of a Killing field defines an infinitesimal minimal deformation. Hence $\chi(X_2) \in \chi(\ker \call J)$, where $\call J$ 
denotes the Jacobi operator, and so, following Ohnita \cite{ohnita} we have
\[
\Delta_L f - (A+2) f = \const = C.
\]
and the pair $(C, f- \bar f) \in \R \oplus \{ f \in C^\infty(L): \Delta_L f =  (A+2) f \}$, where $\bar f =\frac 1 {\vol(L)} \int_L \eta(X)dv$.  
So the first claim follows.

Every $X \in \lie g = \lie{aut}(M)$ defines a trivial deformation of $L$, hence there is a linear map $\alpha: \lie g \rightarrow \Defo(L)$ 
given by $\alpha(X) = \chi(X_2)$. 

Its kernel is $\ker \alpha = \{ X \in \lie g: X|_L \in \Gamma(TL) \} \subseteq \lie{iso}(L)$. So we have

\begin{align} \label{so}
1 + \dim E_{A+2}	& \geq \dim \alpha(\lie g) \\ \nonumber
				& = \dim \lie g - \dim \ker \alpha \\ \nonumber
				& \geq \dim \lie g - \dim \lie{so}(n+1) \\
				&= \dim \lie g - \frac{n(n+1)}{2}. \nonumber
\end{align}

So we have the second claim in the statement.
\end{proof}

Let us specialize to Sasaki-Einstein manifolds and  assume that  $M$ is regular, so it is a principal circle bundle $\pi: M \rightarrow B$ over a K\"ahler-Einstein 
base manifold $B$ and consider the case when the equality holds in the previous theorem. We prove the following, generalizing \cite[Thm.~1.2]{lewang} together with a rigidity result.

\begin{thm}\label{prop:lowerbound}
If $M$ is regular and the eigenvalue $2n+2$ of $\Delta_L$ has multiplicity exactly $ \dim \lie{aut}(M) - \frac 1 2 n(n+1)-1$ then $L$ is totally geodesic in $M$ and $M$ is a principal circle bundle over the complex projective space.
\end{thm}

\begin{proof}
The projection $\tilde L = \pi(L) \subseteq B$ is a Lagrangian submanifold of a K\"ahler-Einstein manifold and it is known that $\tilde L$ is covered by $L$ \cite{reckziegel}.

To have equality one needs to have equality in \eqref{so} so we conclude that the isometry group of $\tilde L$ is the largest possible, i.e. its Lie algebra is $\lie{so}(n+1)$. Let this isometry group be denoted by $K$.
The group $K$, being a subgroup of the Sasaki transformation group of $M$, sends leaves into leaves and thus acts on $B$. We claim that the action has cohomogeneity one.

Indeed it is known, see \cite{transfgroups},  that if $\dim K = \dim \lie{so}(n+1)$ then $L^n$ is either a $n$-sphere or $\RP^n$, written as $\SO(n+1)/H$, where $H= \SO(n)$ or $H= \Z_2 \cdot \SO(n)$ is the stabilizer of a $q \in \tilde L$.
In any case the isotropy representation of $H$ acts transitively on the unit sphere $T_q \tilde L$.  Being $\tilde L$ Lagrangian, the action is transitive also on the unit sphere in the normal space at $q$ and this action has cohomogeneity one, hence also the action of $\SO(n+1)$ on $B$ does.

Let $p \in \tilde L$. Being $\tilde L$ homogeneous under $K$, it is also known from \cite{bedgor} that the orbit $\Omega = K^\C \cdot p$ is open dense in $B$ and Stein, hence in particular affine, and that $B \backslash \Omega$ has complex codimension $1$.

Let $x \in B$ be a principal point. Being $\Omega$ open dense, the $K^\C$-orbit through $x$ is open as well and  intersects $\Omega$, then they coincide. So $B$ is a two-orbit K\"ahler manifold, i.e. is acted on by a complex algebraic group admitting exactly two orbits $\Omega$ and $A$.

They were classified, as complex manifolds, by Ahiezer \cite[Table~2]{ahiezer} in the case of $\Omega$ affine and $A$ of codimension $1$. The occurrences of a group $K$ with Lie algebra $\lie{so}(n+1)$ can be one of the following:

\begin{enumerate}
\item $\tilde L = \SO(n+1)/ \SO(n) = S^n$ and $B = Q_{n}$,
\item $\tilde L = \frac{\SO(n+1)}{\textup{center}}/ S(\O(1) \times \O(n)) = \RP^n$ and $B = \CP^n$,
\item $\tilde L = \Spin(7) / \G_2 = S^7$ and $B = Q_7$,
\item $\tilde L = \SO(7) / \G_2 = \RP^7$ and $B = \CP^7$;
\end{enumerate}
where the projective spaces and the complex hyperquadrics are endowed with the unique K\"ahler-Einstein metric of constant $2n+2$.
This proves that the possible $B$ are only complex hyperquadrics or complex projective spaces and $M$ is a Sasaki-Einstein principal circle bundles over $B$.

Being the pairs in this list symmetric subspaces of $B$, we have that $\tilde L$ is totally geodesic in $B$. By Proposition \ref{prop:reckziegel} of Reckziegel, this is equivalent to say that $L$ is totally geodesic in $M$.  

We want now to exclude the case $B = Q_n$. So far we have the following diagram of immersions and submersions.
\[
\begin{tikzpicture}
\matrix (m) [matrix of math nodes, row sep=6em, column sep=5em,text height=1.5ex, text depth=0.25ex] 
{ (S^n,g)			& 	(M	, g_\textup{SE})			&		(S^{2n+3}, g_c) 			& (\C^{n+2}, \frac 4 c \langle \cdot, \cdot \rangle ) 		\\ 
 (S^n, g)			&	(Q_n, g_Q)					&		(\CP^{n+1}, g^\textup{FS}_c)	 &\\ }; 
\path[right hook->] (m-1-1) edge node[auto] {$\tilde \imath$} (m-1-2);
\path[->](m-1-2) edge node[auto] {$\pi$} (m-2-2); 
\path[->] (m-1-1) edge node[anchor=east] {$=$} (m-2-1);
\path[right hook->] (m-2-1) edge node[auto]{$i$} (m-2-2);
\path[right hook->] (m-2-2) edge node[auto]{$j$} (m-2-3);
\path[->] (m-1-3) edge node[auto] {$p$} (m-2-3);
\path[right hook->] (m-1-3) edge node[auto] {}(m-1-4);
\end{tikzpicture} 
\]

For the metric point of view, we have the Fubini-Study metric $g^\textup{FS}_c$ on $\CP^{n+1}$ 
with constant holomorphic curvature $c$.
This rescaling of the Fubini-Study metric on $\CP^{n+1}$ 
is defined by the metric given by $\frac 4 c$ times the round metric on $S^{2n+3}$, which we denote by $g_c$ \cite[vol.~II, p.~273]{kn}.
The choice of $c$ in $g^\textup{FS}_c$ is such that $g_Q = j^* g^\textup{FS}_c$
is K\"ahler Einstein of Einstein constant $2n+2$ and this happens exactly for $c= \frac{4n+4}{n}$ from \cite{smyth}.

By \cite{ChenNagano} the only totally geodesic spheres in the quadric are immersions $i:x \mapsto [x]$ for $x \in S^n \subset \R^{n+1}$.
The restriction of the quadric metric on it is $\frac{n}{2n+2}$ times the round metric. Being $S^n$ simply connected for $n>1$, we have that the Legendrian $L$ is isometric to its projection in $Q_n$.

Let $\Delta$ be the Laplacian on $S^n$ associated to the metric $\frac{n}{2n+2} g_\textup{round}$. An eigenfunction of $\Delta$ with eigenvalue $2n+2$ is an eigenfunction of the round Laplacian with eigenvalue $n$.

It is known from \cite{spectre} that the round sphere admits the eigenvalue $n$ with multiplicity $n(n+1)$.

To compute the lower bound, we observe that, since every Sasaki automorphisms induces by projection a K\"ahler automorphism of the base, that $\dim \lie{aut}(M) \leq \dim \lie{aut}(B) + 1 = \frac 1 2 (n+2)(n+1) + 1$ since the automorphism group of the hyperquadric is $\SO(n+2)$.

In order not to attain the lower bound we need to have
\[
\dim \lie{aut}(M) < \frac 3 2 n(n+1) + 1
\]
and this is always true since for $n>1$ we have $ \frac 1 2 (n+2)(n+1) + 1 <  \frac 3 2 n(n+1) + 1$.

In the case $n=1$ the quadric $Q_1 = \CP^1$ is a complex projective space, so we are left with the only case $B = \CP^n$. 
\end{proof}
\section{Eigenfunctions using the Nomizu operator} \label{sec:nomizu}

In this section we define another family of eigenfunctions on a Legendrian $L$ of $M$ by making use of the geometry of the K\"ahler cone and its group of K\"ahler automorphisms.

Let $(M, g)$ be a Sasakian manifold of dimension $2n+1$ and let $(C(M), \bar g)$ be its K\"ahler cone.
We let $e_A$ for $A \in \{1, \ldots, 2n+1 \}$ be a local orthonormal frame at some point of $M$ and let $\theta_A$ be its dual.

Then the set $\{ \frac 1 r e_1, \ldots, \frac 1 r e_{2n+1}, \de_r \}$ is an orthonormal frame for the cone metric 
$\bar g = r^2 g + dr^2$ and its dual is $\{ r \theta_1, \ldots, r \theta_{2n+1}, dr \}$.

Let $\bnabla$ be the Levi-Civita connection of the cone metric. From the well known relations \cite{sparks_survey} we have
\begin{align*}
\bnabla \de_r	&= \frac 1 r e_A \otimes \theta_A \\
\bnabla e_B		&= \frac 1 r e_B \otimes dr + \theta_{BC} \otimes e_C - r \de_r \otimes \theta_B.
\end{align*}

\begin{lemma}\label{lemma:extrinsicandintrinsic}
Let $L^n \rightarrow M$ be an immersion  
and let $e_1, \ldots, e_n$ be an orthonormal frame of $L$.
Let $\nabla$ be the Levi Civita connection on $M$. 
Then, for any smooth function $f: M \rightarrow \R$, we have

\begin{equation}\label{equa:extrinsicandintrinsic}
\Delta_L f|_L  = - \sum_{i=1}^n \nabla d f (e_i , e_i )|_L  - H \cdot f |_L  . 
\end{equation}
where $\Delta_L$ is the Hodge Laplacian and $H$ is the mean curvature field of the immersion.

In particular, when the immersion is minimal, we have

\begin{equation}\label{equa:minimalextrinsicandintrinsic}
\Delta_L f |_L =  - \sum_{i=1}^n \nabla d f (e_i , e_i ) |_L  . 
\end{equation}
\end{lemma}
\begin{proof}
 Label as $\nabla^L$ the induced connection on $L$; by definition we have
\begin{align*}
 \sum_{i=1}^n \nabla d f (e_i , e_i ) |_L 	&= \sum_ i e_i e_i f |_L  - \sum_{i=1}^n \nabla_{e_i} e_i f |_L  \\
					&= \sum_{i=1}^n e_i e_i f |_L -  \sum_{i=1}^n \nabla^L_{e_i} e_i f |_L  - \sum_{i=1}^n (\nabla_{e_i} e_i f |_L - \nabla^L_{e_i} e_i f |_L )\\
					&= -\Delta_L f |_L - H \cdot f |_L  ,
\end{align*}
which is precisely the claimed \eqref{equa:extrinsicandintrinsic}. 
Since the assumption on minimality corresponds to the vanishing of $H$, 
we also get the claimed \eqref{equa:minimalextrinsicandintrinsic}.
This completes the proof of the lemma. 
\end{proof}

\begin{lemma}\label{lemma:laplacianofefdoesnotdependonr}
Let $L^n \rightarrow M$ be a \emph{minimal} immersion in a Sasaki manifold.
Let $f$ be a function on the K\"ahler cone $C(M)$ which does not depend on $r$ and let $\Delta_L$ be the Hodge Laplacian on $L$;
finally, let $e_1 , \cdots , e_n$ be an orthonormal frame of L. 
Then we have
\[
\Delta_L f |_L= - \sum_{i=1}^n \bnabla df (e_i, e_i) |_L.
\]
\end{lemma}
\begin{proof}
In view of Lemma \ref{lemma:extrinsicandintrinsic}, 
it suffices to show that for any $i,j \in \{1,\cdots  ,n\}$,
then
\begin{align}
 \bnabla d f (e_i , e_j ) |_L = \nabla d f (e_i , e_j) |_L ,
\end{align}
where as usual $\nabla$ is the  Levi Civita of the Sasaki  metric $g$, 
while $\bnabla$ is the Levi Civita connection of the metric $\bar g = r^2 g + dr^2$.
By the very definition we have
\begin{align*}
 \bnabla d f (e_i , e_j ) |_L   & = e_i e_j f |_L - \bnabla_{e_i} e_j \cdot f |_L \\
				& = e_i e_j f |_L - \bigl ( \nabla_{e_i} e_j \cdot f |_L - \delta_{ij}r\de_r f |_L \bigr ) \\
				& = \nabla d f (e_i , e_j) |_L   ,
\end{align*}
where at the second equality we applied \cite[(1.1)]{sparks_survey}.
This completes the proof of the lemma.
\end{proof}

Let us now construct a family of operators.
For an infinitesimal K\"ahler automorphism $K$ on the cone, i.e. Killing and holomorphic, we define the operator on sections of $TC(M)$ given by
\begin{equation}
M_K = \bnabla K + \frac 1 {2n+2} \div (JK) J.
\end{equation}

\begin{lemma}\label{lemmapropMK}
Let $C(M)$ be the K\"ahler cone over a Sasaki-Einstein manifold and let $K$ as above. Then
\begin{enumerate}[(i)]
\item \label{divJKconst} $\div(JK) = \const$;
\item \label{MJJM} $M_K$ is skew-symmetric and $M_K J = J M_K$;
\item \label{traceJMK} $\tr(JM_K) = 0$;
\item \label{nablaM} $\bnabla M_K = \bar \Rm (\cdot, K)$ where $\bar \Rm$ is the Riemann $(3,1)$-tensor of $\bar g$.
\end{enumerate}
\end{lemma}
\begin{proof}
Let $A_K$ be the associated Nomizu operator, i.e. $A_K = \bnabla K$. Then since $K$ is Killing, its covariant derivative is known to be $\bnabla \,  \bnabla K = \bar \Rm(\cdot, K)$. 
\begin{enumerate}[(i)]

\item Fix $p \in C(M)$ and let $v_i$  be a geodesic frame at $p$ and let $Y$ be any vector field on $C(M)$. Then
\begin{align*}
Y \cdot \div(JK)|_p &= \bar g (\bnabla_Y \bnabla_{v_i} JK, v_i) \\
			&= - \bar g ( \bnabla_Y \bnabla_{v_i} K, J v_i) \\
			&= -\bar g ( (\bnabla_Y A_K) v_i, J v_i)\\
			&= \bar \Rm(Y, K, Jv_i, v_i) \\
			&= 2 \bar \Ric (Y, K)\\
			&= 0
\end{align*}
 since $C(M)$ is Ricci-flat (see \cite{sparks_survey}), where we have used the well known fact that $\Ric(X, Y) = \frac 1 2 \tr (\Rm(X, Y) \circ J)$.
 
 \item Since $K$ is holomorphic it is $\bnabla_{J \cdot} K = J \bnabla_\cdot K$ so $M_K J = J M_K$. 
Since $K$ is Killing, $\bnabla K$ is skew-symmetric and also $J$, so \eqref{MJJM} follows.
 
 \item Let $v_i$ be an orthonormal frame of $C(M)$. Then
 \begin{align*}
 \tr(JM_K) &= \bar g(JM_K v_i, v_i)\\
 		&= \bar g \biggl (\bnabla_{v_i} JK - \frac 1 {2n+2} (\div(JK)) v_i, v_i  \biggr ) \\
		&= \div(JK) - \div(JK) \\
		&= 0.
 \end{align*}
\item By \eqref{divJKconst} and the fact that $J$ is parallel, \eqref{nablaM} follows. \qedhere 
\end{enumerate}
\end{proof}

We will use the following lemma.
\begin{lemma}\label{lemmaRm}
Let $X$ be any field on $M$. Then $\bar \Rm (r \de_r, J r \de_r) K$ and $\bar \Rm (r \de_r, J X) K$ vanish.
\end{lemma}
\begin{proof}
We notice that $\bnabla_{r \de_r} K$ is holomorphic. Indeed, using that $r \de_r$ is holomorphic,  it is
\begin{align*}
\bnabla_{r \de_r} K 	&= [r \de_r, K] + \bnabla_K r\de_r \\
				&= [r \de_r, K] + K
\end{align*}
using that $\bnabla r \de_r = \id$. Hence $\bnabla_{r \de_r}K$  is holomorphic being the sum of two holomorphic fields.
Then we compute
\begin{align*}
\bar \Rm(r \de_r, J r\de_r) K &= \bnabla_{r \de_r} \bnabla_{J r \de_r} K -  \bnabla_{J r \de_r}  \bnabla_{r \de_r} K - \bnabla_{[r \de_r, J r \de_r]} K \\
							&= J \bnabla_{r \de_r} \bnabla_{r \de_r} K - J \bnabla_{r \de_r} \bnabla_{r \de_r} K - \bnabla_{J [r \de_r, r \de_r]} K \\
							&= 0.
\end{align*}
Similarly $\bar \Rm (r \de_r, J X) K = 0$.
\end{proof}

Now consider the family of functions on $f_K : C(M) \rightarrow \mathbb{R}$ defined as
\begin{equation}  \label{equa:defnefK}
f_K = \bar g (M_K \de_r, J \de_r).
\end{equation}
We exploit the fact that $\tr (JM_K) = 0$ for the following lemma, that also uses that $L$ is Legendrian.
\begin{lemma} \label{lemma210}
Let $e_i$ be a frame of the Legendrian $L$. Then
\[
 \sum_{i=1}^n \bar g(M_K e_i, Je_i) = - r^2 f_K.
 \]
\end{lemma}
\begin{proof}
Since $L$ is Legendrian, we can extend $\{ e_i\}$ to an \emph{orthonormal} frame $\{ \frac 1 r e_i, J \frac 1 r e_i, \de_r, \frac 1 r \xi =  J \de_r  \}$ of $C(M)$. Then
\begin{align*}
0 = \tr (J M_K) &=  \frac 1 {r^2} \sum_{i=1}^n \biggl [ \bar g( JM_K e_i, e_i) + \bar g( M_K J e_i, Je_i) \biggr ] \\
                &\phantom{---------}+ \bar g (JM_K J \de_r, J \de_r)  + \bar g(JM_K \de_r, \de_r)
\end{align*}
and from Lemma \ref{lemmapropMK}.\eqref{MJJM} we infer
\[
2 r^2 f + \sum_{i=1}^n 2 \bar g(M_K e_i, J e_i) = 0. \qedhere
\]
\end{proof}

\begin{lemma}\label{lemma:efconstantonar}
For any Killing and holomorphic vector field $K\in \Gamma(TC(M))$, the function $f_K$ is constant along the direction $\de_r$. 
\end{lemma}
\begin{proof}
  Since $\bnabla_{\de_r} \de_r = 0$, we have
\begin{align*}
\de_r f_K &= \bar g( (\bnabla_{\de_r} M_K) \de_r, J \de_r) \\
		&= \frac 1 {r^3} \bar \Rm(r \de_r, K, r \de_r, J r \de_r)\\
		&= - \frac 1 {r^3} \bar \Rm (r \de_r, J r \de_r, K, r \de_r)\\
		&= 0
\end{align*}
by Lemma \ref{lemmaRm}.
\end{proof}

We prove the following.
\begin{thm} \label{thmeigenfunction}
For any Legendrian minimal immersion $L^n \rightarrow M$ in a Sasaki-Einstein manifold, 
and for any both holomorphic and Killing vector field on the K\"ahler cone $K\in \Gamma (TC(M))$, then 
the functions $f_K$ defined by \eqref{equa:defnefK} are eigenfunctions of $\Delta_L$ with eigenvalue $2n+2$.
\end{thm}

\begin{proof}
We fix a vector  field $K$ as in the statement and we set $f= f_K$.
In order to compute $\Delta_L f$, we notice that 
Lemma \ref{lemma:efconstantonar} allows us 
to apply Lemma \ref{lemma:laplacianofefdoesnotdependonr}.
Thus, let $\{ e_1 , \cdots , e_n \}$ be a local frame of $L$.
 
We begin with observing that, for any such vector field $e_i$, 
then there holds
\begin{align}
 \label{equa:derivativeofef}
e_i f = \frac{2}{r} \bar g (M_K \de_r , J e_i).
\end{align}
In fact, 
\begin{align*}
 e_i f & = \bar g ((\bnabla_{e_i} M_K) \de_r , J \de_r) 
         + \bar g ( M_K \bnabla_{e_i} \de_r , J \de_r)
         + \bar g ( M_K \de_r , J \bnabla_{e_i} \de_r )\\
     & = \bar g ( \bar \Rm (e_i, K)\de_r , J\de_r)
         +2 \bar g (M_K \de_r , J \bnabla_{e_i} \de_r)  \\
     & = \frac{2}{r} \bar g (M_K \de_r , J e_i),     
\end{align*}
where at the second equality we applied Lemma \ref{lemmapropMK}.\eqref{MJJM} and \eqref{nablaM}, 
at the third equality we applied Lemma \ref{lemmaRm} and \cite[(1.1)]{sparks_survey}.
Similarly as for \eqref{equa:derivativeofef}, we also get
\begin{align}
 \label{equa:derivativeofef2}
\nabla_{e_i} e_i f =  \frac{2}{r} \bar g (M_K \de_r , J \nabla_{e_i} e_i).
\end{align}
Now we compute
\begin{align*}
 e_i e_i f &= e_i \biggl ( \frac{2}{r} \bar g (M_K \de_r , J e_i) \biggr ) \\
	  &= \frac{2}{r} \biggl (
		  \bar g ((\bnabla_{e_i} M_K) \de_r , J e_i) 
		+ \bar g ( M_K \bnabla_{e_i} \de_r , J e_i)
		+ \bar g ( M_K \de_r , J \bnabla_{e_i} e_i )
			  \biggr )  \\
	  &= \frac{2}{r} \biggl (
		  \bar g (\bar \Rm (e_i, K)\de_r , J e_i)
		+ \frac{1}{r}\bar g ( M_K e_i , J e_i) \\
	  &\phantom{-------}	+ \bar g ( M_K \de_r , J \nabla_{e_i} e_i )
		- \bar g ( M_K \de_r , J r\de_r)
			  \biggr )  \\
	  &= \frac{2}{r^2} \bar g ( M_K e_i , J e_i)
		+ \frac{2}{r} \bar g ( M_K \de_r , J \nabla_{e_i} e_i )
		- 2\bar g ( M_K \de_r , J \de_r),   
\end{align*}
where at the third equality we applied Lemma \ref{lemmapropMK}.\ref{nablaM} and \cite[(1.1)]{sparks_survey}, 
at the third equality we applied Lemma \ref{lemmaRm} and \cite[(1.1)]{sparks_survey}.

 Finally we compute
\begin{align*}
 \Delta_L f |_L	&= -\sum_{i=1}^n \bnabla d f (e_i , e_i)|_L\\
		&= -\sum_{i=1}^n ( e_i e_i f - \bnabla_{e_i}e_i f )|_L\\
		&= -\sum_{i=1}^n ( e_i e_i f - \nabla_{e_i}e_i f + r\de_r f  )|_L \\
		&= -\sum_{i=1}^n \biggl ( \frac{2}{r^2} \bar g ( M_K e_i , J e_i)
		    + \frac{2}{r} \bar g ( M_K \de_r , J \nabla_{e_i} e_i )\\
		&\phantom{------}    
		    - 2\bar g ( M_K \de_r , J \de_r) - \frac{2}{r} \bar g (M_K \de_r , J \nabla_{e_i} e_i) \biggr ) \biggr |_L\\
		&= (2n+2)f|_L,
\end{align*}
where at the third equality we applied \cite[(1.1)]{sparks_survey}, 
at the fourth equality we applied Lemma \ref{lemma:efconstantonar} and \eqref{equa:derivativeofef2},
at the fifth equality we applied Lemma \ref{lemma210}.
This completes the proof of the theorem. 
\end{proof}

\begin{oss}
Let us see how to recover the functions of L\^e-Wang in this setting.
Let $M \in \lie{su}(n+1)$ and consider it as a real $(2n+2) \times (2n+2)$ matrix. It is skew-symmetric and such that $\tr(JM) = 0$.
Consider the vector field on $\C^{n+1}$ given at $x$ by $K_x = Mx$, which is Killing and holomorphic.
We claim that the function $f_K$ is exactly the function $\langle Mx, Jx \rangle$.
Indeed, if $\bnabla$ is the flat connection on $\C^{n+1}$, it is $\bnabla_y K = My$ for $y \in \C^{n+1}$.
Moreover $\div(JK) = \tr(JU) = 0$. So $f_K = \langle Mx, Jx \rangle$ after identifying $x$ with $\de_r|_{(x, 1)}$.
\end{oss}

Let us now see the connection between our two different generalizations.
It is known that there is an inclusion $\lie{aut}(M) \subseteq \lie{aut}(C(M))$ of the algebra of infinitesimal Sasaki automorphisms of $M$ into the algebra of infinitesimal K\"ahler automorphisms of the cone $C(M)$. It consists in seeing a field $V \in \lie{aut}(M)$ trivially extended to the cone and it turns out to be holomorphic and Killing with respect to the cone metric.

We proved in Theorem \ref{thm:eigenf_etaX} that for $X \in \lie{aut}(M)$ the functions on $L$ given by $\eta(X) - \frac 1 {\vol(L)} \int_L \eta(X)dv$ are eigenfunctions of $\Delta_L$ with eigenvalue $2n+2$, in the Sasaki-Einstein assumption.
By seeing $X$ as an infinitesimal K\"ahler automorphism of $C(M)$ we compute
\[
M_X \de_r = \frac 1 r  \biggl [ \bnabla_{r \de_r} X + \frac 1 {2n+2} \div(JX) \xi \biggr ]
\]
and $J \de_r = \frac 1 r \xi$.
Taking their inner product we have
\begin{equation} \label{eq:etaXfX}
f_X = \bar g (M_X \de_r, J \de_r) = \frac 1 {r^2} \bar g(X, \xi) + \frac 1 {2n+2} \div(JX) = \eta(X) +  \frac 1 {2n+2} \div(JX).
\end{equation}

Using Theorem \ref{thm:eigenf_etaX} together with Theorem \ref{thmeigenfunction} we have, after applying the Laplacian to \eqref{eq:etaXfX}, that
\begin{equation}
f_X = \eta(X) - \frac 1 {\vol(L)} \int_L \eta(X)dv.
\end{equation}

Hence our second generalization extends the first.

In the L\^e-Wang setting, we reobtain the fact that $\int_L \eta(X)dv = 0$, which is a fortiori true being $\eta(X)$ an eigenfunction of $\Delta_L$.

\bibliography{../biblio}
\bibliographystyle{amsplain}

\end{document}